\documentclass[12pt]{article}

\usepackage{amsmath, amssymb, amsfonts, amsthm, enumerate, float, lineno, mathrsfs, tikz, url}

\usepackage{caption,fullpage}

\usepackage{multirow}
\usepackage{diagbox}

\usepackage[colorlinks=true,citecolor=black,linkcolor=black,urlcolor=blue]{hyperref}

\theoremstyle{plain}
\newtheorem{theorem}{Theorem}
\newtheorem{lemma}[theorem]{Lemma}

\theoremstyle{definition}
\newtheorem{definition}[theorem]{Definition}

\theoremstyle{remark}

\title{Enumerating Acyclic Digraphs by Descents}
\author{Kassie Archer\thanks{Department of Mathematics, University of Texas at Tyler, karcher@uttyler.edu} \, and Christina Graves\thanks{Department of Mathematics, University of Texas at Tyler, cgraves@uttyler.edu}}
\date{}

\begin{document}

\maketitle

\begin{abstract}
A descent of a labeled acyclic digraph is a directed edge $x\to y$ with $x>y$. In this paper, we find a recurrence for the number of labeled acyclic digraphs with a given number of descents. 

\end{abstract}

\section{Introduction}

\subsection{Background}

An acyclic digraph is a directed graph that has no cycles.  It is well-known that the number of acyclic digraphs on $n$ vertices with labels in $\{1,2,\ldots, n\}$ is given by the following recurrence: 
$$a_n = \sum_{k=1}^n  (-1)^{k+1} {{n}\choose{k}} 2^{k(n-k)}a_{n-k}.$$
The first 8 numbers in the sequence can be found in the last row of the table in Figure \ref{chart}.
This result is attributed to Robinson \cite{R73, R77} and was recovered in the same year by Stanley \cite{S73} who found the following equivalent enumeration for the number of acyclic digraphs: 
$$ \sum_{n=0}^\infty \frac{a_n}{n!2^{\binom{n}{2}}}x^n = \left( \sum_{n=0}^\infty  \frac{ (-1)^n}{n!2^{\binom{n}{2}}}x^n\right)^{-1}.$$

The enumeration of these graphs has been refined with respect to many statistics. These include the number of edges \cite{R92}, the number of sources \cite{L75}, the number of initially connected components \cite{G95}, and the joint distributions of edges, sources, and sinks \cite{G96}. In this paper, we define descents for acyclic digraphs and enumerate these graphs with respect to this newly-defined statistic.

\subsection{Definitions and notation}

Given a labeled acyclic digraph with vertex set $\{1,2,\ldots, n\}$, we let $x \to y$ denote the directed edge from $x$ to $y$. If $x \to y$ is an edge and $x > y$, we call this edge a \emph{descent} while if $x < y$, the edge is an \emph{ascent} or an \emph{increasing} edge. For a vertex $x$ in an acyclic digraph, if there are no edges of the form $y \to x$ for all vertices $y$, then we say vertex $x$ is a \emph{source}.  If there is a directed path from $x$ to $y$, we say the vertex $y$ is \emph{reachable} from $x$. Every vertex is considered to be reachable from itself.

We also make use of some standard notation as follows.  The set $\{1,2,\ldots n\}$ is denoted by $[n]$. Also, we let $\mathbb{Z}_{+}$ denote the set of non-negative integers.  The Gaussian binomial coefficients are used throughout this paper as well.  For $n,j,i \in \mathbb{Z}_+$, the necessary formulas and notation are:
\begin{itemize}
\item $[n]_q! = \frac{1-q^n}{1-q}$;
\item $\displaystyle{{n \choose j}_q = \begin{cases} \frac{[n]_q!}{[j]_q![n-j]_q!} & \text{ if } j \leq n\\ 0 & \text{otherwise;} \end{cases}}$ and
\item $Q_{n,j,i}$ is the coefficient of $q^i$ in ${n \choose j}_q$.
\end{itemize}

It is well-known that $Q_{n,j,i}$ can also be interpreted as the number of partitions of $i$ into $n-j$ parts each less than or equal to $j$. (See for example, \cite{StanleyBook}.)


\begin{figure}
\begin{tabular}{c||r|r|r|r|r|r|r|r}
\diagbox{$k$}{$n$} & 1 & 2 & \multicolumn{1}{c|}{3} & \multicolumn{1}{c|}{4} & \multicolumn{1}{c|}{5} & \multicolumn{1}{c|}{6} & \multicolumn{1}{c|}{7} & \multicolumn{1}{c}{8} \\ \hline \hline
0 & 1 & 2 & 8 & 64 & 1,024 & 32,768 & 2,097,152 & 268,435,456\\
1 & 0 & 1 & 11 & 161 & 3,927 & 172,665 & 14,208,231 & 2,234,357,849\\
2 & 0 & 0 & 5 & 167 & 6,698 & 419,364& 45,263,175 & 8,854,386,165\\
3 & 0 & 0 & 1 & 102 & 7,185 & 656,733& 94,040,848 & 23,016,738,169\\
4 & 0 & 0 & 0 & 39 & 5,477 & 757,939&145,990,526 & 44,953,824,619\\
5 & 0 & 0 & 0 & 9 & 3,107 & 686,425& 181,444,276 & 70,876,002,424\\
6 & 0 & 0 & 0 & 1 & 1,329 & 504,084& 187,742,937 & 94,103,501,133\\
7 & 0 & 0 & 0 & 0 & 423 & 305,207& 165,596,535 & 108,068,923,630\\
8 & 0 & 0 & 0 & 0 & 96 & 153,333& 126,344,492 & 109,265,863,921 \\
9 & 0 & 0 & 0 & 0 & 14 & 63,789& 84,115,442 & 98,446,816,132\\
10 & 0 & 0 & 0 & 0 & 1 & 21,752& 49,085,984 & 79,697,456,418\\
11 & 0 & 0 & 0 & 0 & 0 & 5,959& 25,134,230 & 58,293,422,939\\
12 & 0 & 0 & 0 & 0 & 0 &1,267& 11,270,307 & 38,657,195,560\\
 13 & 0 & 0 & 0 & 0 & 0 & 197 & 4,403,313 & 23,283,565,343\\
 14 & 0 & 0 & 0 & 0 & 0 & 20& 1,486,423 & 12,741,518,134\\
 15 & 0 & 0 & 0 & 0 & 0 & 1& 428,139 & 6,328,700,820\\
 16 & 0 & 0 & 0 & 0 & 0 & 0 & 103,345 & 2,846,683,820\\
 17 & 0 & 0 & 0 & 0 & 0 & 0 & 20,369 & 1,155,387,912 \\
 18 & 0 & 0 & 0 & 0 & 0 & 0 & 3,153 & 421,001,237\\
 19 & 0 & 0 & 0 & 0 & 0 & 0 & 360 & 136,799,627\\
 20 & 0 & 0 & 0 & 0 & 0 & 0 & 27 & 39,294,726\\
 21 & 0 & 0 & 0 & 0 & 0 & 0 & 1 & 9,865,371\\ 
 22 &0 & 0 & 0 & 0 & 0 & 0 & 0 & 2,133,019\\
  23 &0 & 0 & 0 & 0 & 0 & 0 & 0 & 389,396\\
  24 &0 & 0 & 0 & 0 & 0 & 0 & 0 & 58,400\\
  25 &0 & 0 & 0 & 0 & 0 & 0 & 0 & 6,913\\
  26 &0 & 0 & 0 & 0 & 0 & 0 & 0 & 606\\
   27 &0 & 0 & 0 & 0 & 0 & 0 & 0 & 35\\
    28 &0 & 0 & 0 & 0 & 0 & 0 & 0 & 1\\
 \hline
 
 TOTAL & 1 & 3 & 25 & 543 & 29,281 & 3,781,503 & 1,138,779,265 & 783,702,329,343
\end{tabular}
\caption{Values of $d_{n,k}$, the number of acyclic digraphs on $n$ vertices with $k$ descents, for $n \leq 8$.  The total is the number of labeled acyclic digraphs on $n$ vertices.}
\label{chart}
\end{figure}

\section{Main Result}

Our main result gives a recursive formula for the number of acyclic digraphs on $n$ vertices with exactly $k$ descents.  For the remainder of the paper, let $\mathcal{D}_{n,k}$ denote the set of acyclic digraphs on $n$ vertices with $k$ descents.  In order to state our main result, we make use of the following definition.

\begin{definition} Assume $n \geq 1, k \geq 0$, and $m \geq 2$ are integers.

\begin{itemize}
\item Let $a_{n,k,m}$ denote the number of graphs in $\mathcal{D}_{n,k}$ where one of the descents is $m \to 1$.
\item Let $b_{n,k,m}$ denote the number of graphs in $\mathcal{D}_{n,k}$ where $m$ is reachable from 1.
\item Let $c_{n,k,m}$ denote the number of graphs in $\mathcal{D}_{n,k}$ where exactly $m$ of the descents point to 1.
\item Let $d_{n,k}$ denote the number of graphs in $\mathcal{D}_{n,k}$.
\end{itemize}
\end{definition}
Our main result is the value of $d_{n,k}$ stated here in terms of $a_{n,k,m}, b_{n,k,m}$ and $c_{n,k,m}$, which are addressed in Lemmas \ref{lem:a}, \ref{lem:b}, and \ref{lem:c}, respectively. 
The values of $d_{n,k}$ for $n \leq 8$ can be found in Figure \ref{chart}.

\begin{theorem}
The number of acyclic digraphs on $n$ vertices with $k$ descents, denoted $d_{n,k}$, is given by the recurrence
\[ d_{n,k} = 2^{n-1}d_{n-1, k} + (n-1)d_{n,k-1} -  \sum_{m=2}^{n} (a_{n,k-1, m} + b_{n, k-1,m}) - \sum_{m=2}^k (m-1)c_{n, k, m} \]
with initial conditions
\[ d_{n,0} = 2^{n \choose 2} \quad \text{and } \quad d_{0, k} = \begin{cases} 1 & \text{ for } k=0\\ 0 & \text{ for }k \geq 1. \end{cases}  \]
\end{theorem}

\begin{proof}
Consider the base case where $k=0$. The number of acyclic digraphs with 0 descents is found by including any increasing edge. There are a total of ${n \choose 2}$ increasing edges, so the result holds.  For the remainder of the proof, assume $k \geq 1$.

We first note that any acyclic digraph with $k$ descents either has a descent of the form $x\to1$ or it does not.  If it does not have a descent of the form $x\to 1$, then vertex 1 is a source. The number of acyclic digraphs with $k$ descents where 1 is a source is counted recursively by taking any acyclic digraph with $k$ descents on $n-1$ vertices labeled $\{2, 3, \ldots, n\}$ and adding the vertex labeled 1.  Since 1 is smaller than the labels of all the other vertices, we can add any increasing edge of the form $1\to y$ for any $y \in\{2,3,\ldots, n\}$ without creating a descent. 
Thus, there are a total of 
\[ 2^{n-1}d_{n-1, k} \]
acyclic digraphs on $n$ vertices with $k$ descents where the vertex labeled 1 is a source.

The remainder of the proof counts the number of acyclic digraphs on $n$ vertices with $k$ descents where 1 is not a source.  Consider the set $\mathcal{D}_{n,k-1}$ of acyclic digraphs on $n$ vertices with only $k-1$ descents.  For each graph in $\mathcal{D}_{n,k-1}$ and for each $m$ between 2 and $n$, we want to add the descent $m \to 1$ 
to create a acyclic digraph on $n$ vertices with $k$ descents. However, this new graph with the added descent is only in $\mathcal{D}_{n,k}$ if:
\begin{itemize}
\item the graph did not already have the descent $m \to 1$ and
\item the graph did not have a path from $1$ to $m$.
\end{itemize}
The second condition is necessary to ensure that our new graph remains acyclic.  Thus, we can add the descent $m \to 1$ to a total of
\[ d_{n, k-1} - a_{n, k-1, m} - b_{n,k-1, m} \]
graphs.  Summing over all possible $m$ between 2 and $n$ yields a total of
\[ (n-1)d_{n,k-1} -  \sum_{m=2}^{n} (a_{n,k-1, m} + b_{n, k-1,m}). \]
graphs formed in this manner. 

However, counting the desired graphs in such a way counts some graphs more than once, in particular those with more than one descent of the form $x\to1$.  In fact, for any $\ell$ between 2 and $k$, graphs with exactly $\ell$ descents pointing at 1 are counted exactly $\ell$ times.  Subtracting the number of graphs that were counted multiple times yields the desired result.
\end{proof}

\section{Enumeration Lemmas}

The remainder of this paper is then devoted to finding formulas for
\[ \sum_{m=2}^n a_{n,k,m}, \sum_{m=2}^n b_{n,k,m}, \text{ and } \sum_{m=2}^k (m-1)c_{n, k, m}.\]

To this end, we define two more values.

\begin{definition}\label{tu} Assume $n \geq 1$ and $k \geq 0$ are integers.
\begin{itemize}
\item Let $t_{n,k}$ denote the number of graphs in $\mathcal{D}_{n,k}$ where every vertex is reachable from 1.
\item Let $u_{n,k}$ denote the number of  graphs in $\mathcal{D}_{n,k}$ where every vertex is reachable from $n$.
\end{itemize}
\end{definition}

In order to find formulas for $t_{n,k}$ and $u_{n,k}$ we state a brief lemma which will be used later.

\begin{lemma}\label{q}
There are $Q_{n,j,i}$ ways to partition $[n]$ into two sets $X$ and $Y$ where $|X| = j$ and with
$i$ pairs $(x,y) \in X \times Y$ with $x < y$.
\end{lemma}

\begin{proof}
Consider a partition of $[n]$ into two sets $X$ and $Y$ where $|X| = j$ and with
$i$ pairs $(x,y) \in X \times Y$ with $x < y$.  Write $Y = \{y_1, y_2, \ldots, y_{n-j}\}$ where $y_{r} < y_{r+1}$ for $r \in [n-j-1]$.  Notice that $0 \leq y_r - r \leq j$ for all $r \in [n-j]$.  Also, the number of pairs $(x,y) \in X \times Y$ with $y < x$ is 
\[ i = (y_1 - 1) + (y_2 - 2) + \cdots + (y_{n-j} - n+j).\]
Because each $y_r - r$ is between 0 and $j$, this directly corresponds to a partition of $i$ into $n-j$ parts each less than or equal to $j$. 

\end{proof}

The formulas for $t_{n,k}$ and $u_{n,k}$ are described in the following lemma.

\begin{lemma}\label{lem:tu} Let $t_{n,k}$ and $u_{n,k}$ be as defined in Definition \ref{tu}. Then $t_{n,k}$ and $u_{n,k}$ satisfy the following recurrences:
\[  t_{n,k} = \sum_{(j,i,r,s) \in \Omega_t}  u_{j,s}t_{n-j,r} {i \choose k-s-r} 2^{(j-1)(n-j)-i}(2^{n-j}-1) Q_{n-2,j-1,i} \]
and 
\[ u_{n,k} =   \sum_{(j,i,r,s) \in \Omega_u} t_{j,s}u_{n-j,r} \left({i \choose k-r-s}  - {i-n+j \choose k-s-r}\right) 2^{j(n-j)-i}Q_{n-2,j-1,i-n+1} \]
where
\[ \Omega_t = \{(j,i,r,s) \in \mathbb{Z}_{+} ^4 | 1 \leq j \leq n-1, \ i \leq (j-1)(n-j-1),\  r +s \leq k\} \] and
\[ \Omega_u = \{(j,i,r,s) \in \mathbb{Z}_{+}^4 | 1 \leq j \leq n-1, \  n-1 \leq i \leq j(n-j), \ r+ s \leq k-1\},\]
with initial conditions
\[ t_{n,0} = \begin{cases}1 & \text{ for } n=0\\ [n-1]_2! & \text{ for } n \geq 1, \end{cases} \quad t_{1,k} = \begin{cases} 1 & \text{ for } k = 0\\ 0 & \text{ for } k \geq 1, \end{cases} \]
and
\[ u_{n,0} = \begin{cases} 1 & \text{ for } n=1\\ 0 & \text{ for } n \neq 1. \end{cases} \]

\end{lemma}

\begin{proof}
We first consider the formula for $t_{n,k}$.  For the base case where $k=0$, we need to find the number of graphs on $n$ vertices with 0 descents where every vertex is reachable from 1.  The only way the vertex labeled 2 is reachable from 1 is if the edge $1 \to 2$ is included.  Let $2 \leq m \leq n$ be another vertex.  If vertices $\{1,2,\ldots m-1\}$ are reachable from 1, then vertex $m$ is reachable from 1 if at least one edge of the form $m' \to m$ is included for some $1 \leq m' \leq m-1$; thus, there are $2^{m-1} - 1$ possible edges that can point to $m$.  Multiplying over all $m$ between 2 and $n$ yields a total of
\[ (2^1 - 1)(2^2 - 1) \cdots (2^{n-1} - 1) = [n-1]_2! \]
such graphs.  For the base case where $n=1$, it is clear that there is only one acyclic digraph and it has 0 descents.
 
For the remainder of the proof of the formula for $t_{n,k}$, assume that $n,k \geq 1$. The set of all graphs in $\mathcal{D}_{n,k}$ where every vertex is reachable from 1 can be partitioned based on the number of vertices, $j$, that are reachable from $n$.  For any such graph, let $X$ be the set of vertices that are reachable from $n$ and $Y = [n] \setminus X$.  Thus we want to count how many graphs there are in $\mathcal{D}_{n,k}$ where every vertex is reachable from 1 that also satisfy the following conditions:
\begin{itemize}
\item $|X| = j$,
\item the number of pairs of vertices $(x,y) \in X \times Y$ with $x < y$ is $i$,
\item the number of descents in the subgraph induced by $X$ is $s$, and
\item the number of descents in the subgraph induced by $Y$ is $r$,
\end{itemize}
where the values of $j, i, r,$ and $s$ satisfy certain conditions.  We first notice that  $n \in X$ since $n$ is reachable from itself.  Also, since $n$ is reachable from 1 and the desired graphs are acyclic, 1 cannot be reachable from $n$ and thus $1 \in Y$.  It follows that $1 \leq j \leq n-1$. Furthermore, the number of pairs of vertices $(x,y) \in X \times Y$ with $x < y$ is at most $(j-1)(n-j-1)$ which occurs when every vertex in $X \setminus \{n\}$ is smaller than every vertex in $Y \setminus \{1\}$. Finally, we note that $0 \leq r+s \leq k$. Thus, $(j,i,r,s) \in \Omega_t$ as defined in the statement of Lemma \ref{lem:tu}.

Since $n \in X$ and $1 \in Y$, we now consider the remaining vertices. The number of ways to partition the remaining $n-2$ vertices into two sets $X$ and $Y$ with $|X| = j$ and with $i$ pairs of vertices $(x,y) \in X \times Y$ with $x < y$ is $Q_{n-2, j-1, i}$ by Lemma \ref{q}.  There are $u_{j,s}$ choices for the subgraph induced by $X$ and $t_{n-j,r}$ choices for the subgraph induced by $Y$.  Also, because the graphs must have a total of $k$ descents, the remaining $k-s-r$ descents can be chosen from the $i$ pairs of vertices $(x,y) \in X \times Y$ where $x < y$.  Since every vertex is reachable from 1, vertex $n$ must be reachable from 1; thus there must be an edge from some vertex in $Y$ to $n$.  There are $n-j$ possible increasing edges from $Y$ to $n$, and at least one must be included yielding a total of $2^{n-j} - 1$ possibilities.  Finally, there are $(n-j)(j-1) - i$ possible increasing edges from $Y \setminus \{1\}$ to $X \setminus \{n\}$.  Because these edges can all be included or not, we multiply our total by $2^{(n-j)(j-1) - i}$. Note that we cannot include any edges from $X$ to $Y$ since all vertices reachable from the vertex labeled $n$ are already in $X$.

We use a similar technique to find a formula for $u_{n,k}$. The set of graphs in $\mathcal{D}_{n,k}$ where every vertex is reachable from $n$ can be partitioned based on the number of vertices, $j$, that are reachable from 1. For any such graph, let $X$ be the set of vertices that are reachable from $1$ and $Y = [n] \setminus X$. Again, we want to count the number of graphs that satisfy the aforementioned conditions along with the following:
\begin{itemize}
\item $|X| = j$
\item the number of pairs of vertices $(x,y) \in X \times Y$ with $x < y$ is $i$,
\item the number of descents in the subgraph induced by $X$ is $s$, and
\item the number of descents in the subgraph induced by $Y$ is $r$, 
\end{itemize}
where $j, i, r,$ and $s$ satisfy certain conditions.  Because every vertex is reachable from $n$, and the vertex labeled 1 is reachable from itself, we have that $1 \in X$ and $n \in Y$ and hence $1 \leq j \leq n-1$. The number of pairs of vertices $(x,y) \in X \times Y$ with $x <y$ is at least $n-1$ since there are $j$ such pairs of the form $(x,n)$ and $n-j$ pairs of the form $(1,y)$. The pair $(n,1)$ is counted twice in this argument and hence $i \geq n-1$. The largest number of pairs occurs when every element in $Y$ is greater than every element in $X$ and thus $i \leq j(n-j)$. Finally, it is clear that $r+s \leq k-1$, since there must be at least one descent of the form $y\to 1$ where $y\in Y$ and $1 \in X$. Thus, we see that $(j,i,r,s) \in \Omega_u$ as defined in Lemma \ref{lem:tu}.

Consider the number of ways to partition the remaining $n-2$ vertices into $X$ and $Y$ meeting the desired conditions.  We know that the number of pairs of vertices $(x,y) \in X \times Y$ with $x < y$ is $i$, but $j$ of these pairs are of the form $(x, n)$ and $n-j$ pairs are of the form $(1,y)$. So there are $i - (n-1)$ pairs of vertices in $(x,y) \in (X \setminus \{1\}) \times (Y \setminus \{n\})$ with $x < y$. Thus, the number of ways to partition the remaining $n-2$ vertices into sets $X$ and $Y$ meeting the desired conditions is $Q_{n-2,j-1, i-n+1}$.

The remainder of the terms in our recursive formula for $u_{n,k}$ can be seen in a very similar manner to that of $t_{n,k}$.  There are $t_{j,s}$ and $u_{n-j,r}$ choices for the subgraphs induced by $X$ and $Y$ respectively.  In order to get a total of $k$ descents, the remaining $k-r-s$ descents can be chosen from the $i$ pairs of vertices $(x,y) \in X\times Y$ with $x<y$. However, because every vertex is reachable from $n$, at least one of those $k-r-s$ descents must be of the form $y\to1$. There are $n-j$ pairs of the form $(1,y)$, and hence there are $i-n+j$ pairs that do not contain $1$. Thus, the term
\[ {i \choose k-r-s}  - {i-n+j \choose k-s-r} \]
counts the number of ways the $k-r-s$ descents can be chosen from the $i$ pairs of vertices while still including at least one descent pointing at 1. Finally, there are $j(n-j) - i$ possible increasing edges from $Y$ to $X$ which can all be included or not which multiplies our total by $2^{j(n-j) - i}$.

\end{proof}

We are now ready to state the formulas needed for our main result, namely \[ \sum_{m=2}^n a_{n,k,m}, \sum_{m=2}^n b_{n,k,m}, \text{ and } \sum_{m=2}^k (m-1)c_{n, k, m},\] which are found in the lemmas below.

\begin{lemma}\label{lem:a} For $2 \leq m\leq n$, let $a_{n,k,m}$ denote the number of graphs in $\mathcal{D}_{n,k}$ where one of the descents is $m \to 1$. Then
\[ \sum_{m=2}^k a_{n,k,m} = \sum_{(j,i,r,s,\ell) \in \Omega_a} \ell \cdot d_{n-j,r}t_{j,s}{n-j \choose \ell}{i \choose k-s-r-\ell} 2^{(j-1)(n-j) - i} Q_{n-1,j-1,i} \]
where 
\[ \Omega_a = \{(j,i,r,s,\ell) \in \mathbb{Z}_{+}^5 | 1 \leq j \leq n-1, \ i \leq (j-1)(n-j), \ r + s \leq k-1, \ 1 \leq \ell \leq k-r-s\} \]

\end{lemma}

\begin{proof}
We begin by partitioning the set of all graphs in $\mathcal{D}_{n,k}$ that have \emph{at least one} descent of the form $y\to 1$ by the number of vertices, $j$, that are reachable from 1.  For any such graph, let $X$ be the set of vertices that are reachable from 1 and let $Y = [n]\setminus X$.  We will proceed by counting the number of graphs satisfying the stated conditions along with the following:

\begin{itemize}
\item $|X| = j$,
\item the number of pairs of vertices $(x,y) \in (X \setminus \{1\}) \times Y$ where $x < y$ is $i$,
\item the number of descents in the subgraph induced by $X$ is $s$,
\item the number of descents in the subgraph induced by $Y$ is $r$, and
\item the number of descents pointing at 1 is $\ell$,
\end{itemize}
where $(j,i,r,s,\ell)$ satisfy certain conditions.  Notice that 1 is reachable from itself and thus $1 \in X$.  Also, $Y$ cannot be empty because there must be at least one descent pointing at 1. Thus, $1 \leq j \leq n-1$, $r+s \leq k-1$, and $1 \leq \ell \leq k-r-s$. Also, the maximum number of pairs of vertices $(x,y) \in (X \setminus \{1\}) \times Y$ where $x < y$ is $(n-j)(j-1)$ which occurs when every element in $Y$ is greater than every element of $X$.  Hence, $(j,i,r,s,\ell) \in \Omega_a$ as defined in the statement of the Lemma.

The number of ways to partition the remaining $n-1$ vertices into sets $X$ and $Y$ meeting the desired conditions is $Q_{n-1,j-1,i}$ by Lemma \ref{q}. There are $d_{n-j,r}$ and $t_{j,s}$ choices for the subgraphs induced by $Y$ and $X$, respectively, and there are ${n-j \choose \ell}$ ways to choose the $\ell$ descents pointing at 1. The remaining $k-s-r-\ell$ descents are can be chosen from the $i$ pairs of vertices $(x,y) \in (X \setminus \{1\}) \times Y$ where $x < y$. Finally, there are $(j-1)(n-j) - i$ pairs of vertices $(x,y) \in X \times Y$ where $x > y$; any of these increasing edges can be included. Thus, there are
\[ \sum_{(j,i,r,s,\ell) \in \Omega_a} d_{n-j,r}t_{j,s}{n-j \choose \ell}{i \choose k-s-r-\ell} 2^{(j-1)(n-j) - i} Q_{n-1,j-1,i} \]
graphs in $\mathcal{D}_{n,k}$ that that have \emph{at least one} descent pointing at 1. However, for each $1 \leq \ell \leq k-r-s$, we have contributed to the sum $\sum_{m=2}^k a_{n,k,m}$ exactly $\ell$ times.  Thus the equality stated in Lemma \ref{lem:a} holds.

\end{proof}

\begin{lemma}\label{lem:b}
For $2\leq m\leq n$, let $b_{n,k,m}$ denote the number of graphs in $\mathcal{D}_{n,k}$ where $m$ is reachable from 1. Then,
\[ \sum_{m=2}^k b_{n,k,m} = \sum_{(j,i,r,s) \in \Omega_b} (j-1)d_{n-j,r}t_{j,s}{i \choose k-r-s} 2^{j(n-j) - i} Q_{n-1,j-1,i-n+j} \]
where 
\[ \Omega_b = \{(j,i,r,s) \in \mathbb{Z}_{+}^4 | 2 \leq j \leq n,\  n-j \leq i \leq j(n-j), \ r + s \leq k\} \]

\end{lemma}

\begin{proof}
We partition the graphs in $\mathcal{D}_{n,k}$ by the number of vertices, $j$ that are reachable from 1 where $2 \leq j \leq n$.  If we can count the graphs in $\mathcal{D}_{n,k}$ where $j$ vertices are reachable from 1, then multiplying by $(j-1)$ gives the number of graphs in $\mathcal{D}_{n,k}$ where $m$ is reachable from 1 and and there are exactly $j$ vertices reachable from 1 for all $2 \leq m \leq n$.  Summing over all $j$ will then give the desired result.

Toward this end, we again let $X$ be the set of vertices reachable from 1 and let $Y = [n] \setminus X$.  We count the number of graphs with the conditions that:
\begin{itemize}
\item $|X| = j$,
\item the number of pairs of vertices $(x,y) \in X \times Y$ where $x < y$ is $i$,
\item the number of descents in the subgraph induced by $X$ is $s$, and
\item the number of descents in the subgraph induced by $Y$ is $r$,
\end{itemize}
where $j,i,r$ and $s$ satisfy certain conditions.  In particular, we need $2 \leq j \leq n$ and $0 \leq r+s \leq k$. Also, because all $n-j$ elements in $Y$ are greater than 1, we see that $i \geq n-j$, and the maximum value of $i$ occurs when every element in $X$ is greater than every element in $Y$ which gives $i \leq j(n-j)$. Hence, $(j,i,r,s) \in \Omega_b$ as defined in the statement of Lemma \ref{lem:b}.

Consider the number of ways to partition the set of vertices $[n] \setminus \{1\}$ into $X$ and $Y$ meeting the desired conditions.  We know that the number of pairs of vertices $(x,y) \in X \times Y$ with $x < y$ is $i$, but $n-j$ pairs are of the form $(1,y)$. So there are $i - (n-1)$ pairs of vertices in $(x,y) \in (X \setminus \{1\}) \times Y$ with $x < y$. Thus, the number of ways to partition the remaining $n-1$ vertices into sets $X$ and $Y$ meeting the desired conditions is $Q_{n-1,j-1, i-n+j}$ by Lemma \ref{q}.  It is clear that the number of choices for the subgraph induced by $X$ is $t_{j,s}$, and that the number of choices for the subgraph induced by $Y$ is $d_{n-j,r}$. The remaining $k$ descents can be chosen from the $i$ pairs of vertices $(x,y) \in X \times Y$ where $x < y$, and the increasing edges from $Y$ to $X$, of which there are $j(n-j) - i$, can be included or not.  The result follows.
\end{proof}

\begin{lemma}\label{lem:c} For $2\leq m\leq n$, let $c_{n,k,m}$ denote the number of graphs in $\mathcal{D}_{n,k}$ where exactly $m$ of the descents are of the form $x\to1$.  Then
\[ \sum_{m=2}^k (m-1)c_{n,k,m} \]
is equivalent to
\[  \sum_{(m,j,i,r,s) \in \Omega_c}  (m-1)d_{n-j,r}t_{j,s}{n-j \choose m}{i \choose k-r-s-m} 2^{(j-1)(n-j) - i} Q_{n-1,j-1,i} \]
where 
\[ \Omega_c = \{(m,j,i,r,s) \in \mathbb{Z}_{+}^5 | 2 \leq m \leq k,\  1 \leq j \leq n-m, \ i \leq (j-1)(n-j), \ r + s \leq k-m\}. \]
\end{lemma}

\begin{proof} For a fixed $2 \leq m \leq k$, partition all graphs in $\mathcal{D}_{n,k}$ that have exactly $m$ descents pointing at 1 by the number of vertices, $j$, that are reachable from 1. In this case we have $1 \leq j \leq n-m$ as there must be at least $m$ vertices that are not reachable from 1.  Let $X$ be the set of vertices reachable from 1 and let $Y=[n] \setminus X$. We count the number of graphs satisfying the following conditions:
\begin{itemize}
\item $|X| = j$,
\item the number of pairs of vertices $(x,y) \in (X \setminus \{1\}) \times Y$ where $x < y$ is $i$,
\item the number of descents in the subgraph induced by $X$ is $s$, and
\item the number of descents in the subgraph induced by $Y$ is $r$,
\end{itemize}
where the values $j,i,s$, and $r$ satisfy certain conditions. Since 1 is reachable from itself, clearly $j\geq 1$. Since there are $m$ descents pointing at 1, these $m$ elements are not reachable from 1, and thus must be elements of $Y$. Thus, $j\leq n-m$. As before, there are at most $(j-1)(n-j)$ possible edges which could be descents of the form $y\to x$ where $y \in Y$ and $x \in (X\setminus{1})$. Finally, since there are $m$ descents of the form $y\to 1$ for some $y\in Y$, there are at most $k-m$ descents which occur in the subgraphs induced by $X$ and by $Y$, thus $r+s\leq k-m$. We conclude that 
$(m,j,i,r,s) \in \Omega_c$.

By Lemma \ref{q}, there are $Q_{n-1,j-1,i}$ ways to partition the vertices $[n] \setminus \{1\}$ with these given conditions. Also, there are $t_{j,s}$ and $d_{n-j,r}$ choices for the subgraphs induced by $X$ and $Y$ respectively. Of the $n-j$ vertices in $Y$, exactly $m$ of them must point at 1 (giving the ${n-j \choose m}$ term) and the remaining $k-r-s-m$ descents can be chosen from the $i$ pairs of vertices $(x,y) \in (X \setminus \{1\}) \times Y$ where $x < y$ (which gives the ${i \choose k-r-s-m}$ term). Finally, there are $(j-1)(n-j)-i$ edges from $Y$ to $X$ which are increasing that can also be added without introducing any new descents or cycles. The result follows.
\end{proof}

\bibliographystyle{amsplain}
\bibliography{DAGbib.bib}

\end{document}